\documentclass[10pt]{amsart}
\usepackage{amsmath}
\usepackage{amssymb}
\usepackage{enumerate}
\usepackage{amsbsy}
\usepackage{amsfonts}
\usepackage{color}

\newtheorem{thm}{Theorem}[section]

\newtheorem{lem}[thm]{Lemma}

\newtheorem{defn}[thm]{Definition}

\numberwithin{equation}{section}

\theoremstyle{remark}
\newtheorem{rem}{Remark}

\newcommand{\al}{\alpha}

\newcommand{\ld}{\lambda}

\begin{document}

\newcommand{\sar}{S_\al(\mathbb R)}
\newcommand{\xar}{X_\al(\mathbb R)}
\newcommand{\yar}{Y_\al(\mathbb R)}

\newcommand{\ald}{\mathcal A}

\newcommand{\blb}{\big[}
\newcommand{\brb}{\big]}
\newcommand{\re}{\mbox{Re}}
\newcommand{\mlf}{\mathbf{m}_{1,\ld}}
\newcommand{\mlft}{\widetilde{\mathbf{m}}_{1, \ld}}
\newcommand{\mls}{\mathbf{m}_{2, \ld}}
\newcommand{\im}{\mbox{Im}}
\newcommand{\bl}{\big\langle}
\newcommand{\br}{\big\rangle}

\newcommand{\hlp}{H_{\mathbf L,\; p}}
\newcommand{\hl}{H_{\mathbf L}}
\newcommand{\hj}{H_{\mathbf J}}

\title[Inhomogeneous cubic-quintic NLS]
      {Well-posedness and scattering of inhomogeneous cubic-quintic NLS}

\author[Y. Cho]{Yonggeun Cho}
\address{Department of Mathematics, and Institute of Pure and Applied Mathematics, Chonbuk National University, Jeonju 561-756, Republic of Korea}
\email{changocho@jbnu.ac.kr}

\begin{abstract}
In this paper we consider inhomogeneous cubic-quintic NLS in space dimension $d = 3$:
$$
iu_t = -\Delta u + K_1(x)|u|^2u + K_2(x)|u|^4u.
$$
We study local well-posedness, finite time blowup, and small data scattering and non-scattering for the ICQNLS when $K_1, K_2 \in C^4(\mathbb R^3 \setminus \{0\})$ satisfy growth condition $|\partial^j K_i(x)| \lesssim |x|^{b_i-j}\, (j = 0, 1, 2, 3, 4)$ for some $b_i \ge 0$ and for $x \neq 0$. To this end we use the Sobolev inequality for the functions $f \in H^n \,(n = 1, 2)$ such that $\||\mathbf L|^\ell f\|_{H^n} < \infty \,(\ell = 1, 2)$, where $\mathbf L$ is the angular momentum operator defined by $\mathbf L = x \times (-i\nabla)$.
\end{abstract}

\thanks{2010 {\it Mathematics Subject Classification.} 35Q55, 35Q53. }
\thanks{{\it Key words and phrases.} inhomogeneous NLS, well-posedness, finite time blowup, small data scattering,  Sobolev inequality, angular momentum condition}

\maketitle

\newcommand{\na}{|\nabla|^\al}
\section{Introduction}
In this paper we consider the following Cauchy problem for inhomogeneous cubic-quintic nonlinear Schr\"odinger equations of the form:
\begin{align}\label{main eqn}\left\{\begin{array}{l}
i\partial_tu = -\Delta u + K_1(x)Q_1(u) + K_2(x)Q_2(u)\;\;
\mbox{in}\;\;\mathbb{R}^{1+3},
\\
u(x,0) = \varphi(x),
\end{array}\right.
\end{align}
where $Q_1(u) = |u|^2u, Q_2(u) = |u|^4u$, and $K_l \in C^4(\mathbb R^3 \setminus \{0\}; \mathbb C)$. The model of ICQNLS \eqref{main eqn} can be a dilute BEC when both the two- and three-body interactions of the condensate are considered. For this see \cite{bpvt, ts} and the references therein. Also it has been considered to study the laser guiding in an axially nonuniform plasma channel. For this see \cite{gill, ss, ts}.

In this paper we consider ICQNLS with $K_l$ satisfying the growth condition: for some constants $b_1, b_2 \ge 0$
\begin{align}\label{ass-k}
|\partial^j K_l| \lesssim |x|^{b_l - j}, \,j = 0, 1, 2, 3, 4,\, l = 1, 2,
\end{align}
where $\partial$ is one of the partial derivatives $\partial_j, j = 1, 2, 3$.
Some basic notations are listed at the end of this section.

By Duhamel's formula, \eqref{main eqn} is written as an
integral equation
\begin{equation}\label{integral}
u = e^{it\Delta}\varphi -i \int_0^t e^{i(t-t')\Delta}[K_1(x)Q_1(u(t')) + K_2(x)Q_2(u(t'))]\,dt'.
\end{equation}
Here we define the linear propagator $e^{it\Delta}$ given by the linear problem $i\partial_t v =  -{\Delta} v$ with initial datum $v(0) = f$. It is
formally given  by
\begin{align}\label{int eqn}
e^{it{\Delta}}f = \mathcal F^{-1} (e^{-it|\xi|^2} \mathcal F(f)) = (2\pi)^{-3}\int_{\mathbb{R}^3} e^{i( x\cdot \xi - t|\xi|^2 )}\widehat{f}(\xi)\,d\xi,
\end{align}
where $\widehat{f} = \mathcal F( f)$ denotes the Fourier transform of $f$ and $\mathcal F^{-1}$ the inverse Fourier transform such that
$$\mathcal{F}(f)(\xi) = \int_{\mathbb{R}^3} e^{- ix\cdot \xi} f(x)\,dx,\quad \mathcal F^{-1} (g)(x) = (2\pi)^{-d}\int_{\mathbb{R}^3} e^{ix\cdot \xi} g(\xi)\,d\xi.$$

If $K_l$ are real-valued, then we can define mass and energy of the solution $u$ of \eqref{main eqn} as follows:
\begin{align*}
m(u(t)) &:= \|u(t)\|_{L_x^2}^2,\\
E(u(t)) &:= \frac12\|\nabla u\|_{L_x^2}^2 + \frac14\int K_1(x)|u(t, x)|^4\,dx + \frac16\int K_2(x)|u(t, x)|^6\,dx.
\end{align*}
We say that the mass and the energy of solutions are conserved if they are constant with respect to time.

The aim of this paper is to establish a well-posedness theorey, a finite time blowup, and a scattering theory for suitable growth rate $b_1, b_2 \ge 0$. In case that $K_l$ are radially symmetric, the authors \cite{che, cg, zhu} considered well-posedness, finite time blowup and stability of radial solutions. The main obstacle of that problems is the growth of $K_l$ at infinity. To avoid this Sobolev inequalities of radial $H^1$ functions were utilized. However, nothing in general cases has been known about the global behavior like scattering as far as we know. For other work treating bounded or decaying coefficients like $|x|^{-b}$ see \cite{fiwa, lww, mer, rs} or \cite{fuoh, coge, far}, respectively.

 To circumvent the lack of symmetry of $K_l$ and growth at space infinity, we suggest alternatives of radial symmetry, the angular momentum conditions, for which we introduce the angular momentum operator $\mathbf L$:
$$
\mathbf L = (L_1, L_2, L_3) = x \times (-i\nabla).
$$
It is well-known that $|\mathbf L|^2 (\,:= \mathbf L \cdot \mathbf L = \sum_{j = 1, 2, 3}L_j^2) = -\Delta_{S^2}$, where $\Delta_{S^2}$ is the Laplace-Beltrami operator on the unit sphere. Now we define Sobolev spaces $H_{\mathbf L,\; p}^{n,\; \ell} (n, \ell = 0, 1, 2, 1 \le p \le \infty)$ associated with $\mathbf L$ as follows:
\begin{align*}
\hlp^{0,\; \ell} &= L_x^p,\quad \hlp^{1,\; 0} = H_p^1,\\
H_{\mathbf L,\; p}^{n, 1} &= \{f \in \hlp^{n-1,\; 2} \cap H^n : \|f\|_{H_{\mathbf L, p}^{n,\; 1}} := \|f\|_{H_p^n} + \|f\|_{\hlp^{n-1,\; 2}} +  \|\mathbf L f\|_{H_p^n} < \infty\},\\
H_{\mathbf L,\; p}^{n,\; 2} &= \{f \in \hlp^{n,\; 1} : \|f\|_{H_{\mathbf L, p}^{n,\; 2}} := \|f\|_{H_{\mathbf L,\; p}^{n,\; 1}} + \||\mathbf L|^2 f\|_{H_p^n} < \infty\},
\end{align*}
Here $H_p^n$ denotes the standard $L^p$ Sobolev space. If $p = 2$, then we drop $p$ and denote $ H_{\mathbf L,\; 2}^{n, \;\ell}$ by $\hl^{n, \;\ell}$ . These spaces give us Sobolev type inequalities associated angular momentum such as $\||x|^bf\|_{L_x^\infty} \lesssim \|f\|_{\hl^{1, 1}}$ for $0 < b < 1$ and $\||x|f\|_{L_x^\infty} \lesssim \|f\|_{\hl^{1,\; 2}}$ (see Lemma \ref{ang-sobo} below).
%%%%%%%%%%%%%%%%%%%%%%%%%%%%%%%%%%%%%%%%%%%%%%%%%%%%%%%%%%%%%%%%%%%%%%%%%%%%%%%%%%%%%%%%%%%%%%%%%%%%%%%%%%%%%%%%%%%%%%%%%%%%%%%%%%%%%%%%%%%%%%%%%%%%%%%%%%%%%%%%%%%%%%%%%%%%%%%%%%%%%
%%%%%%%%%%%%%%%%%%%%%%%%%%%%%%%%%%%%%%%%%%%%%%%%%%%%%%%%%%%%%%%%%%%%%%%%%%%%%%%%%%%%%%%%%%%%%%%%%%%%%%%%%%%%%%%%%%%%%%%%%%%%%%%%%%%%%%%%%%%%%%%%%%%%%%%%%%%%%%%%%%%%%%%%%%%%%%%%%%%%%

%%%%%%%%%%%%%%%%%%%%%%%%%%%%%%%%%%%%%%%%%%%%%%%%%%%%%%%%%%%%%%%%%%%%%%%%%%%%%%%%%%%%%%%%%%%%%%%%%%%%%%%%%%%%%%%%%%%%%%%%%%%%%%%%%%%%%%%%%%%%%%%%%%%%%%%%%%%%%%%%%%%%%%%%%%%%%%%%%%%%%
%%%%%%%%%%%%%%%%%%%%%%%%%%%%%%%%%%%%%%%%%%%%%%%%%%%%%%%%%%%%%%%%%%%%%%%%%%%%%%%%%%%%%%%%%%%%%%%%%%%%%%%%%%%%%%%%%%%%%%%%%%%%%%%%%%%%%%%%%%%%%%%%%%%%%%%%%%%%%%%%%%%%%%%%%%%%%%%%%%%%%

%%%%%%%%%%%%%%%%%%%%%%%%%%%%%%%%%%%%%%%%%%%%%%%%%%%%%%%%%%%%%%%%%%%%%%%%%%%%%%%%%%%%%%%%%%%%%%%%%%%%%%%%%%%%%%%%%%%%%%%%%%%%%%%%%%%%%%%%%%%%%%%%%%%%%%%%%%%%%%%%%%%%%%%%%%%%%%%%%%%%%
Our first result is on the local well-posedness, whose definition is the following.
\begin{defn}\label{lwp defn}
The equation \eqref{main eqn} is said to be locally well-posed $\hl^{n, \;\ell}$ if there exist maximal existence time interval $I_* = (-T_*, T^*)$ and a unique solution $u \in C(I_*, \hl^{n,\; \ell})$ with continuous dependency on the initial data and blowup alternative $(T^* < \infty \Rightarrow \lim_{t \to T^*}\|u(t)\|_{\hl^{n,\; \ell}} = \infty)$.
\end{defn}
%%%%%%%%%%%%%%%%%%%%%%%%%%%%%%%%%%%%%%%%%%%%%%%%%%%%%%%%%%%%%%%%%%%%%%%%%%%%%%%%%%%%%%%%%%%%%%%%%%%%%%%%%%%%%%%%%%%%%%%%%%%%%%%%%%%%%%%%%%%%%%%%%%%%%%%%%%%%%%%%%%%%%%%%%%%%%%%%%%%%%

%%%%%%%%%%%%%%%%%%%%%%%%%%%%%%%%%%%%%%%%%%%%%%%%%%%%%%%%%%%%%%%%%%%%%%%%%%%%%%%%%%%%%%%%%%%%%%%%%%%%%%%%%%%%%%%%%%%%%%%%%%%%%%%%%%%%%%%%%%%%%%%%%%%%%%%%%%%%%%%%%%%%%%%%%%%%%%%%%%%%%
%%%%%%%%%%%%%%%%%%%%%%%%%%%%%%%%%%%%%%%%%%%%%%%%%%%%%%%%%%%%%%%%%%%%%%%%%%%%%%%%%%%%%%%%%%%%%%%%%%%%%%%%%%%%%%%%%%%%%%%%%%%%%%%%%%%%%%%%%%%%%%%%%%%%%%%%%%%%%%%%%%%%%%%%%%%%%%%%%%%%%
\begin{thm}\label{lwp}
$(1)$ If $b_1 = b_2 = 0$, then \eqref{main eqn} is locally well-posed in $H^1$.\\
$(2)$ If $n-1 < b_1 < n$ and $0 \le b_2 < n+2$ for $n = 1, 2$, then \eqref{main eqn} is locally well-posed in $\hl^{n,\; 1}$.\\
$(3)$ If $b_1 = n$ and $0 \le b_2 \le n+2$ for $n = 1, 2$, then \eqref{main eqn} is locally well-posed in $\hl^{n,\; 2}$. \\
$(4)$ If $K_l$ are real-valued, then in any cases the mass and the energy are conserved.
\end{thm}
%%%%%%%%%%%%%%%%%%%%%%%%%%%%%%%%%%%%%%%%%%%%%%%%%%%%%%%%%%%%%%%%%%%%%%%%%%%%%%%%%%%%%%%%%%%%%%%%%%%%%%%%%%%%%%%%%%%%%%%%%%%%%%%%%%%%%%%%%%%%%%%%%%%%%%%%%%%%%%%%%%%%%%%%%%%%%%%%%%%%%
We prove this theorem via standard contraction mapping theorem. If $b_1 \le 2, b_2 \le 4$, we can control the growing coefficients by using Sobolev inequality associated with angular momentum. For example we need to estimate $\||x|^{b_1}|\partial u|^2|\mathbf L u|\|_{L_x^2}$, which can be done by the bound $\||x|^{\frac{b_1}2}u\|_{L_x^\infty}^2\|\mathbf L u\|_{L_x^2} \lesssim \|u\|_{\hl^{n, \ell}}^3$. In case that $b_1 > 2$ we cannot control it only with Sobolev inequality. To this end one can try to show the local well-posedness for the initial data with higher regularity and additional weight condition $(|x|^k \varphi \in L_x^2, k = 1, 2, \cdots)$. We will not pursue this issue here. The local well-posedness results are far away from the sharpness of regularity on the space and angle. One may improve them via fractional Sobolev space and fractional Leibniz rule \cite{fgo}.
%%%%%%%%%%%%%%%%%%%%%%%%%%%%%%%%%%%%%%%%%%%%%%%%%%%%%%%%%%%%%%%%%%%%%%%%%%%%%%%%%%%%%%%%%%%%%%%%%%%%%%%%%%%%%%%%%%%%%%%%%%%%%%%%%%%%%%%%%%%%%%%%%%%%%%%%%%%%%%%%%%%%%%%%%%%%%%%%%%%%%
%%%%%%%%%%%%%%%%%%%%%%%%%%%%%%%%%%%%%%%%%%%%%%%%%%%%%%%%%%%%%%%%%%%%%%%%%%%%%%%%%%%%%%%%%%%%%%%%%%%%%%%%%%%%%%%%%%%%%%%%%%%%%%%%%%%%%%%%%%%%%%%%%%%%%%%%%%%%%%%%%%%%%%%%%%%%%%%%%%%%%

%%%%%%%%%%%%%%%%%%%%%%%%%%%%%%%%%%%%%%%%%%%%%%%%%%%%%%%%%%%%%%%%%%%%%%%%%%%%%%%%%%%%%%%%%%%%%%%%%%%%%%%%%%%%%%%%%%%%%%%%%%%%%%%%%%%%%%%%%%%%%%%%%%%%%%%%%%%%%%%%%%%%%%%%%%%%%%%%%%%%%
%%%%%%%%%%%%%%%%%%%%%%%%%%%%%%%%%%%%%%%%%%%%%%%%%%%%%%%%%%%%%%%%%%%%%%%%%%%%%%%%%%%%%%%%%%%%%%%%%%%%%%%%%%%%%%%%%%%%%%%%%%%%%%%%%%%%%%%%%%%%%%%%%%%%%%%%%%%%%%%%%%%%%%%%%%%%%%%%%%%%%
The next result is on the finite time blowup when the initial energy is negative.
\begin{thm}\label{blowup}
Let $K_l$ be real-valued function such that $K_1 - x\cdot \nabla K_1 \le \alpha K_1$ and $4K_2 - x\cdot \nabla K_2 \le \alpha K_2$ for some $\alpha \ge 0$. Let $u$ be the local solution of \eqref{main eqn} as in Theorem \ref{lwp} with $|x|\varphi \in L_x^2$. Suppose that  $E(\varphi) < 0$. Then the solution blows up in finite time.
\end{thm}
If $K_1 = -|x|^{-b_1}$ and $K_2 = -|x|^{b_2}$, then the condition on $K_l$ implies that $b_1 \le \alpha +1$ and $b_2 \le \alpha +4$. We use the standard virial argument for which the weight condition $|x|\varphi \in L_x^2$ and the sign condition of the coefficients $K_l$ are necessary. Once a regular solution exists even for $b_1 > 2$, the finite time blowup can be shown by the same argument.
%%%%%%%%%%%%%%%%%%%%%%%%%%%%%%%%%%%%%%%%%%%%%%%%%%%%%%%%%%%%%%%%%%%%%%%%%%%%%%%%%%%%%%%%%%%%%%%%%%%%%%%%%%%%%%%%%%%%%%%%%%%%%%%%%%%%%%%%%%%%%%%%%%%%%%%%%%%%%%%%%%%%%%%%%%%%%%%%%%%%%
%%%%%%%%%%%%%%%%%%%%%%%%%%%%%%%%%%%%%%%%%%%%%%%%%%%%%%%%%%%%%%%%%%%%%%%%%%%%%%%%%%%%%%%%%%%%%%%%%%%%%%%%%%%%%%%%%%%%%%%%%%%%%%%%%%%%%%%%%%%%%%%%%%%%%%%%%%%%%%%%%%%%%%%%%%%%%%%%%%%%%

%%%%%%%%%%%%%%%%%%%%%%%%%%%%%%%%%%%%%%%%%%%%%%%%%%%%%%%%%%%%%%%%%%%%%%%%%%%%%%%%%%%%%%%%%%%%%%%%%%%%%%%%%%%%%%%%%%%%%%%%%%%%%%%%%%%%%%%%%%%%%%%%%%%%%%%%%%%%%%%%%%%%%%%%%%%%%%%%%%%%%
%%%%%%%%%%%%%%%%%%%%%%%%%%%%%%%%%%%%%%%%%%%%%%%%%%%%%%%%%%%%%%%%%%%%%%%%%%%%%%%%%%%%%%%%%%%%%%%%%%%%%%%%%%%%%%%%%%%%%%%%%%%%%%%%%%%%%%%%%%%%%%%%%%%%%%%%%%%%%%%%%%%%%%%%%%%%%%%%%%%%%
Now we consider a small data scattering.
\begin{defn}\label{scattering defn}
We say that a solution $u$ to \eqref{main eqn} scatters (to $u_\pm$) in a Hilbert space $\mathcal H$ if there exist  $\varphi_\pm \in \mathcal H$ $($with $u_\pm(t) = e^{it{\Delta}}\varphi_\pm)$ such that $\lim_{t  \to \pm\infty}\|u(t) - u_\pm\|_{\mathcal H} = 0$.
\end{defn}
%%%%%%%%%%%%%%%%%%%%%%%%%%%%%%%%%%%%%%%%%%%%%%%%%%%%%%%%%%%%%%%%%%%%%%%%%%%%%%%%%%%%%%%%%%%%%%%%%%%%%%%%%%%%%%%%%%%%%%%%%%%%%%%%%%%%%%%%%%%%%%%%%%%%%%%%%%%%%%%%%%%%%%%%%%%%%%%%%%%%%

%%%%%%%%%%%%%%%%%%%%%%%%%%%%%%%%%%%%%%%%%%%%%%%%%%%%%%%%%%%%%%%%%%%%%%%%%%%%%%%%%%%%%%%%%%%%%%%%%%%%%%%%%%%%%%%%%%%%%%%%%%%%%%%%%%%%%%%%%%%%%%%%%%%%%%%%%%%%%%%%%%%%%%%%%%%%%%%%%%%%%
%%%%%%%%%%%%%%%%%%%%%%%%%%%%%%%%%%%%%%%%%%%%%%%%%%%%%%%%%%%%%%%%%%%%%%%%%%%%%%%%%%%%%%%%%%%%%%%%%%%%%%%%%%%%%%%%%%%%%%%%%%%%%%%%%%%%%%%%%%%%%%%%%%%%%%%%%%%%%%%%%%%%%%%%%%%%%%%%%%%%%
Our small data scattering is the following.
\begin{thm}\label{scattering}
Let $0 \le b_1 < \frac23$ and $0 \le b_2 < \frac83$.  If $\|\varphi\|_{H_{\mathbf L}^{1, 1}}$ is sufficiently small, then there exists a unique $u \in (C \cap L^\infty)(\mathbb R ; H_{\mathbf L}^{1, \;1})$ to \eqref{main eqn} and $u^\pm \in H_{\mathbf L}^{1, \;1}$ to which $u$ scatters in $H_{\mathbf L}^{1, \;1}$.
\end{thm}
%%%%%%%%%%%%%%%%%%%%%%%%%%%%%%%%%%%%%%%%%%%%%%%%%%%%%%%%%%%%%%%%%%%%%%%%%%%%%%%%%%%%%%%%%%%%%%%%%%%%%%%%%%%%%%%%%%%%%%%%%%%%%%%%%%%%%%%%%%%%%%%%%%%%%%%%%%%%%%%%%%%%%%%%%%%%%%%%%%%%%
%%%%%%%%%%%%%%%%%%%%%%%%%%%%%%%%%%%%%%%%%%%%%%%%%%%%%%%%%%%%%%%%%%%%%%%%%%%%%%%%%%%%%%%%%%%%%%%%%%%%%%%%%%%%%%%%%%%%%%%%%%%%%%%%%%%%%%%%%%%%%%%%%%%%%%%%%%%%%%%%%%%%%%%%%%%%%%%%%%%%%
For the proof we carry out nonlinear estimates with constants not depending on the local time. This is possible due to the endpoint Strichartz estimates and Sobolev inequality associated with angular momentum, when $b_l$ are small enough not to make nonlinearity super-critical in energy. One can study this type result for a general nonlinearity $|x|^b|u|^\alpha u (2 < \alpha < 4)$, for which see the nonlinear estimate in Remark \ref{gen-case} below.
%%%%%%%%%%%%%%%%%%%%%%%%%%%%%%%%%%%%%%%%%%%%%%%%%%%%%%%%%%%%%%%%%%%%%%%%%%%%%%%%%%%%%%%%%%%%%%%%%%%%%%%%%%%%%%%%%%%%%%%%%%%%%%%%%%%%%%%%%%%%%%%%%%%%%%%%%%%%%%%%%%%%%%%%%%%%%%%%%%%%%
%%%%%%%%%%%%%%%%%%%%%%%%%%%%%%%%%%%%%%%%%%%%%%%%%%%%%%%%%%%%%%%%%%%%%%%%%%%%%%%%%%%%%%%%%%%%%%%%%%%%%%%%%%%%%%%%%%%%%%%%%%%%%%%%%%%%%%%%%%%%%%%%%%%%%%%%%%%%%%%%%%%%%%%%%%%%%%%%%%%%%

%%%%%%%%%%%%%%%%%%%%%%%%%%%%%%%%%%%%%%%%%%%%%%%%%%%%%%%%%%%%%%%%%%%%%%%%%%%%%%%%%%%%%%%%%%%%%%%%%%%%%%%%%%%%%%%%%%%%%%%%%%%%%%%%%%%%%%%%%%%%%%%%%%%%%%%%%%%%%%%%%%%%%%%%%%%%%%%%%%%%%
%%%%%%%%%%%%%%%%%%%%%%%%%%%%%%%%%%%%%%%%%%%%%%%%%%%%%%%%%%%%%%%%%%%%%%%%%%%%%%%%%%%%%%%%%%%%%%%%%%%%%%%%%%%%%%%%%%%%%%%%%%%%%%%%%%%%%%%%%%%%%%%%%%%%%%%%%%%%%%%%%%%%%%%%%%%%%%%%%%%%%
If $b_1$ is big, then we expect a non-scattering. Here we give a sufficient condition as follows.
\begin{thm}\label{nonscattering}
Assume that $K_1(x) = |x|^{b_1}$ and $K_2(x) = |x|^{b_2}$ for $b_1, b_2 > 0$, and $\lambda \in \mathbb R$.  Let $u$ be a smooth global solution of \eqref{main eqn} with $b_1 \ge 2$ and $0 < b_2 < 3+b_1$, which scatters to $u_\pm = e^{it\Delta}\varphi_\pm$ in $L_x^2$ for some smooth function $\varphi_\pm$. Then $u, u_\pm \equiv 0$.
\end{thm}
%%%%%%%%%%%%%%%%%%%%%%%%%%%%%%%%%%%%%%%%%%%%%%%%%%%%%%%%%%%%%%%%%%%%%%%%%%%%%%%%%%%%%%%%%%%%%%%%%%%%%%%%%%%%%%%%%%%%%%%%%%%%%%%%%%%%%%%%%%%%%%%%%%%%%%%%%%%%%%%%%%%%%%%%%%%%%%%%%%%%%
For the proof we use pseudo-conformal identity to get the potential energy bound $\frac14\int|x|^{b_1}|u|^4\,dx + \frac16\int|x|^{b_2}|u|^6\,dx \lesssim t^{b_1-3}$, which is crucial to the estimate of quintic term. Theorem \ref{nonscattering} implies that the scattering in the sense of Definition \ref{scattering defn} does not occur in the long-range case $b_1 \ge 2$. We think the case $b_1 = 2$ will be borderline of the scattering and non-scattering. In this critical case it is highly expected that a modified scattering will occur. This will be another interesting issue to be pursued. The scattering problem still remains open in short-range cases $\frac{2}3 \le b_1 < 2$. This short range together with critical case may be taken into account by utilizing the generator of Galilean transformation $\mathbf J$ (see \eqref{j} below).
%%%%%%%%%%%%%%%%%%%%%%%%%%%%%%%%%%%%%%%%%%%%%%%%%%%%%%%%%%%%%%%%%%%%%%%%%%%%%%%%%%%%%%%%%%%%%%%%%%%%%%%%%%%%%%%%%%%%%%%%%%%%%%%%%%%%%%%%%%%%%%%%%%%%%%%%%%%%%%%%%%%%%%%%%%%%%%%%%%%%%
%%%%%%%%%%%%%%%%%%%%%%%%%%%%%%%%%%%%%%%%%%%%%%%%%%%%%%%%%%%%%%%%%%%%%%%%%%%%%%%%%%%%%%%%%%%%%%%%%%%%%%%%%%%%%%%%%%%%%%%%%%%%%%%%%%%%%%%%%%%%%%%%%%%%%%%%%%%%%%%%%%%%%%%%%%%%%%%%%%%%%

%%%%%%%%%%%%%%%%%%%%%%%%%%%%%%%%%%%%%%%%%%%%%%%%%%%%%%%%%%%%%%%%%%%%%%%%%%%%%%%%%%%%%%%%%%%%%%%%%%%%%%%%%%%%%%%%%%%%%%%%%%%%%%%%%%%%%%%%%%%%%%%%%%%%%%%%%%%%%%%%%%%%%%%%%%%%%%%%%%%%%
%%%%%%%%%%%%%%%%%%%%%%%%%%%%%%%%%%%%%%%%%%%%%%%%%%%%%%%%%%%%%%%%%%%%%%%%%%%%%%%%%%%%%%%%%%%%%%%%%%%%%%%%%%%%%%%%%%%%%%%%%%%%%%%%%%%%%%%%%%%%%%%%%%%%%%%%%%%%%%%%%%%%%%%%%%%%%%%%%%%%%
This paper is organized as follows: In Section 2 we introduce angular Sobolev inequality and some properties of angular momentum operators. We give a proof for Theorem \ref{lwp} in Section 3 by standard contraction argument and for Theorem \ref{blowup} in Section 4 via virial argument. In Sections 5, 6 we prove small data scattering, Theorem \ref{scattering} and non-scattering, Theorem \ref{nonscattering}.
%%%%%%%%%%%%%%%%%%%%%%%%%%%%%%%%%%%%%%%%%%%%%%%%%%%%%%%%%%%%%%%%%%%%%%%%%%%%%%%%%%%%%%%%%%%%%%%%%%%%%%%%%%%%%%%%%%%%%%%%%%%%%%%%%%%%%%%%%%%%%%%%%%%%%%%%%%%%%%%%%%%%%%%%%%%%%%%%%%%%%
%%%%%%%%%%%%%%%%%%%%%%%%%%%%%%%%%%%%%%%%%%%%%%%%%%%%%%%%%%%%%%%%%%%%%%%%%%%%%%%%%%%%%%%%%%%%%%%%%%%%%%%%%%%%%%%%%%%%%%%%%%%%%%%%%%%%%%%%%%%%%%%%%%%%%%%%%%%%%%%%%%%%%%%%%%%%%%%%%%%%%

\medskip

\noindent\textbf{Basic notations.}\\
%\noindent$\bullet$ Littlewood-Paley operators: $\dot \beta \in C_{0, rad}^\infty$ with $\dot \beta \ddot \beta = \dot \beta$ and $\ddot \beta(\xi) = \dot\beta(\xi/2) + \dot\beta(\xi) + \dot\beta(2\xi)$. Let $\mathcal F (\dot P_Nf)(\xi) = \dot\beta(\xi/N)\widehat{f}$ for any dyadic number $N > 0$ and $\ddot{P}_N = P_{N/2} + P_N + P_{2N}$. Then $\dot P_N\ddot{P}_N = \dot P_N$. We denote $\sum_{N \le N_0}\dot P_N$ and $\dot P_N$ for $N > N_0$ by $P_{N_0}$ and $P_N$, respectively.
%.\\

\noindent$\bullet$ Fractional derivatives: $D^s = (-\Delta)^\frac{s}2 = \mathcal F^{-1}|\xi|^s\mathcal F$, $\Lambda^s = (1- \Delta)^\frac{s}2 = \mathcal F^{-1}(1+|\xi|^2)^\frac s2\mathcal F$ for $s > 0$.\\

\noindent$\bullet$ Function spaces:  $\dot H_r^s = D^{-s}L^r$, $\dot H^s = \dot H_2^s$, $H_r^s = \Lambda^{-s} L^r$, $H^s = H_2^s$, $L^r = L_x^r(\mathbb R^d)$ for $s \in \mathbb R$ and $1 \le r \le \infty$.  \\

\noindent$\bullet$ Mixed-normed spaces: For a Banach space $X$, $u \in L_I^q X$ iff $u(t) \in X$ for a.e. $t \in I$ and $\|u\|_{L_I^qX} := \|\|u(t)\|_X\|_{L_I^q} < \infty$. Especially, we denote  $L_I^qL_x^r = L_t^q(I; L_x^r(\mathbb R^d))$, $L_{I, x}^q = L_I^qL_x^q$ and $L_t^qL_x^r = L_{\mathbb R}^qL_x^r$.\\
%$$X_\al(I) := L_I^6L_x^\frac{2d}{d-\al/3},\quad Y_\al(I) := L_I^2L_x^\frac{2d}{d+\al}.$$

%\noindent$\bullet$ $\big\langle\;,\big\rangle$ is the complex inner
%product in $L^2$. For vector-valued functions $\bl \mathbf f\,;\, \mathbf g \br = \sum_{1 \le j \le n}\bl f_j, g_j \br$ for $\mathbf f = (f_1,\cdots, f_n), \mathbf g = (g_1, \cdots, g_n)$.\\

%\noindent$\bullet$ $\big[\mathcal S, \mathcal T \big]$ denotes the commutator $\mathcal S\mathcal T - \mathcal T\mathcal S$ for any operators $\mathcal S$ and $\mathcal T$ defined on suitable Banach spaces.\\

%\noindent$\bullet$ We use complex inner products such as $\big<u, v\big> = \int_{\mathbb R^d} u \,\overline v\, dx$ and $\big<f\,;\, g \big> = \sum_{1 \le j \le d}\big<f_j, g_j \big>$ for $f = (f_1,\cdots, f_d), g = (g_1, \cdots, g_d)$.\\

\noindent$\bullet$ As usual different positive
constants depending are denoted by the same letter $C$, if not specified. $A \lesssim B$ and $A \gtrsim B$ means that $A \le CB$ and
$A \ge C^{-1}B$, respectively for some $C>0$. $A \sim B$ means that $A \lesssim B$ and $A \gtrsim B$.\\

\section{Useful lemmata}

If a pair $(q, r)$ satisfies that $2 \le q, r \le \infty$, $\frac 2{q} + \frac 3{r} = \frac 32$, then it is said to be {\it admissible}.

\begin{lem}[\cite{kt}]\label{str}
Let $(q, r)$ and $(\widetilde q, \widetilde r)$ be any admissible pair.
Then we have
\begin{align*}
\|e^{it{\Delta}} \varphi\|_{L_t^{q} L_x^r} &\lesssim \| \varphi\|_{L_x^2},\\
\|\int_0^t e^{i(t-t')\Delta}F\,dt'\|_{L_t^qL_x^r} &\lesssim \|F\|_{L_t^{\widetilde q'}L_x^{\widetilde r'}}.
\end{align*}
\end{lem}

\begin{lem}\label{hs-ineq}
For any $f \in \dot H_p^s(\mathbb R^3) (1 < p < \infty, 0 < s < \frac{3}p)$ we have
$$
\||x|^{-s}f\|_{L_x^p} \lesssim \|f\|_{ \dot H_p^s}.
$$
\end{lem}
\begin{proof}
This can be done by interpolation between Theorem 2 of \cite{ms} and critical Sobolev inequality $\|f\|_{BMO} \lesssim \|f\|_{\dot H^\frac np}$ ( For instance see  \cite{tr}).
\end{proof}

\begin{lem}[\cite{crt-n}]\label{sobo-sp}
For any smooth function $f$ there holds
$$
\|\mathbf Lf\|_{L_x^2} \sim \||\mathbf L|f\|_{L_x^2} = \|(-\Delta_{S^2})^\frac12 f\|_{L_x^2},\quad \sum_{1 \le j, k \le 3}\|L_j L_k f\|_{L_x^2}^2 \sim \||\mathbf L|^2f\|_{L_x^2}^2.
$$
\end{lem}

\begin{lem}[\cite{choz1, fawa}]\label{ang-sobo}
Let $0 < b < 1$. Then for any $f \in H_{\mathbf L}^{1,\; 1}$ there holds
$$
\||x|^bf\|_{L_x^\infty} \lesssim \|f\|_{H_{\mathbf L}^{1,\; 1}}.
$$
And also for any $f \in \hl^{1, 2}$
$$
\||x|f\|_{L_x^\infty} \lesssim \|f\|_{\hl^{1, \;2}}.
$$
\end{lem}

\begin{lem}\label{ang-decay}
Let $0 < b < 1$, $0 <  \varepsilon < 1 - b$, and $2 \le p < \infty$. Then for any $f \in H_{\mathbf L}^{1,\; 1} \cap L_x^p$ we have
$$
\||x|^b f\|_{L_x^\frac{p}\varepsilon} \lesssim \|f\|_{H_{\mathbf L}^{1,\; 1}}^{1-\varepsilon}\|f\|_{L_x^p}^\varepsilon.
$$
\end{lem}

\begin{proof}
Since $\varepsilon < 1 - b$, $\frac{b}{1-\varepsilon} < 1$ and thus we get from Lemma \ref{ang-sobo} that
$$
\||x|^b f\|_{L_x^\frac p\varepsilon} \le \||x|^\frac{b}{1-\varepsilon} f\|_{L_x^\infty}^{1-\varepsilon}\|f\|_{L_x^p}^\varepsilon \le \|f\|_{H_{\mathbf L}^{1, 1}}^{1-\varepsilon}\|f\|_{L_x^p}^\varepsilon.
$$
\end{proof}

By direct calculation we have the following.
\begin{lem}\label{conv-radial}
$(1)$ Let $s \ge 0$. Then $\mathbf L D^s f = D^s \mathbf L f$ and $\mathbf \Lambda^s f = \Lambda^s \mathbf L f$ for any smooth function $f$.\\
$(2)$ Let $\psi$ be smooth and radially symmetric. Then $$\mathbf L (\psi * f) = \psi * (\mathbf L f).$$
\end{lem}

%The final one is on the Sobolev inequality on the unit sphere.
%\begin{lem}\label{sob-s2}
%Let $1 < p < \infty$
%$$
%\|\mathbf L f\|_{H_p^1} \sim \|f\|_{L_x^p} + \sum_{j = 1, 2, 3}\|\partial_j \mathbf L f\|_{L_x^p}.
%$$
%\end{lem}

\newcommand{\xl}{X_{\mathbf L}}
\newcommand{\xloz}{X_{\mathbf L}^{1,\; 0}}
\newcommand{\xloo}{X_{\mathbf L}^{1,\; 1}}
\newcommand{\ist}{I_T}
\newcommand{\xlot}{X_{\mathbf L}^{1,\; 2}}
\newcommand{\xltt}{X_{\mathbf L}^{2,\; 2}}
\newcommand{\xlto}{X_{\mathbf L}^{2,\; 1}}

\section{Local well-posedness: Proof of Theorem \ref{lwp}}\label{sec4}
Let $I_T = [-T, T]$. Let us define a complete metric spaces $\xl^{n,\; \ell}(T, \rho), n, \ell = 0, 1, 2$ with metric $d^{n, \;\ell}$  by
\begin{align*}
\xl^{1,\; 0}(T, \rho) &:= \Big\{u \in S_T^{1,\; 0} = L_{\ist}^{10}H_\frac{30}{13}^1 \cap (C \cap L^\infty)(I_T;  H^1) : \|u\|_{S_T^{1,\; 0}} \le \rho \Big\}, \;\;d^{1, \; 0}(u ,v) = \|u - v\|_{S_T^{1,\; 0}},\\
\xl^{n,\; \ell}(T, \rho) &:= \Big\{u \in S_T^{n,\; \ell} = L_{\ist}^{10}H_{\mathbf L,\; \frac{30}{13}}^{n,\;\ell} \cap (C \cap L^\infty)(I_T;  H_{\mathbf L}^{n, \;\ell}) : \|u\|_{S_T^{n, \ell}} \le \rho \Big\}, \;\;d^{n,\; \ell}(u ,v) = \|u - v\|_{S_T^{n,\; \ell}}.
\end{align*}

From the assumption \eqref{ass-k} it follows that for each $j = 0, 1, 2$
\begin{align}\label{k}
|\partial^j K_l(x)| + |\partial^j \mathbf L K_l(x)| + |\partial^j |\mathbf L|^2 K_l(x)| \lesssim |x|^{b_l - j}.
\end{align}

We will show that  the nonlinear functional $\Psi(u) = e^{it\Delta}\varphi + \mathcal N(u)$ is a contraction on $\xl^{n, \;\ell}(T, \rho)$ for each case. Here
$$
\mathcal N (u) = -i\int_0^t e^{i(t-t')\Delta} [K_1Q_1(u) + K_2Q_2(u)]\,dt'.
$$
By $N_l^{n,\; \ell}$ we denote the derivatives of Duhamel part as follows:
\begin{align*}
N_l^{n,\; 0} &= -i\partial^n \int_0^t e^{i(t-t')\Delta} [K_lQ_l(u)]\,dt',\\
N_l^{n,\; 1} &= -i\partial^n \mathbf L\int_0^t e^{i(t-t')\Delta} [K_lQ_l(u)]\,dt',\\
N_l^{n,\; 2} &= -i\partial^n |\mathbf L|^2\int_0^t e^{i(t-t')\Delta} [K_lQ_l(u)]\,dt'.
\end{align*}
We have by Leibniz rule and Lemma \ref{conv-radial} that for $n = 1, 2$
\begin{align*}
N_l^{n,\; 0} &= -i\sum_{k = 0}^n\left(
                                \begin{array}{c}
                                  n \\
                                  k \\
                                \end{array}
                              \right)
\int_0^t e^{i(t-t')\Delta}[(\partial^{n-k}  K_l)\partial^k Q_l(u), \\
N_l^{n,\; 1} &= -i\sum_{k = 0}^n\left(
                                \begin{array}{c}
                                  n \\
                                  k \\
                                \end{array}
                              \right)
\int_0^t e^{i(t-t')\Delta}[(\partial^{n-k} \mathbf L K_l)\partial^k Q_l(u) + (\partial^{n-k} K_l) \partial^k\mathbf L Q_l(u)]\,dt',\\
N_l^{n,\; 2} &= -i\sum_{k = 0}^n\left(
                                \begin{array}{c}
                                  n \\
                                  k \\
                                \end{array}
                              \right)\int_0^t e^{i(t-t')\Delta}[(\partial^{n-k} |\mathbf L|^2 K_l)\partial^k Q_l(u) + 2(\partial^{n-k}\mathbf L K_l)\cdot \partial^k \mathbf L Q_l(u) + (\partial^{n-k}  K_l)\partial^k |\mathbf L|^2 Q_l(u)]\,dt'.
\end{align*}

\subsection{Case: $b_1 = b_2 = 0$}\label{case0}
Given $\rho$, it follows from Lemmas \ref{str} and \ref{hs-ineq} that for any $u \in \xloz(T, \rho)$
\begin{align*}
\|N_1^{1,\; 0}\|_{L_{\ist}^{10}L_x^\frac{30}{13} \cap L_{I_T}^\infty L_x^2} &\lesssim \||x|^{-1}|u|^3\|_{L_{\ist}^2 L_x^\frac65} + \||u|^2|\partial u|\|_{L_{\ist}^2 L_x^\frac65}\\
&\lesssim T^\frac12\Big(\|u\|_{L_{\ist}^\infty L_x^6}^2\|\frac{|u|}{|x|}\|_{L_{\ist}^\infty L_x^2} + \|u\|_{L_{\ist}^\infty L_x^6}^2\|\partial u\|_{L_{\ist}^\infty L_x^2 }\Big)\\
&\lesssim T^\frac12\|u\|_{L_{\ist}^\infty H^{1}}^3 \lesssim T^\frac12\rho^3.
\end{align*}
As for $N_2^{1,\; 0}$ we have
\begin{align*}
\|N_2^{1,\, 0}\|_{L_{\ist}^{10}L_x^\frac{30}{13} \cap L_{I_T}^\infty L_x^2} &\lesssim \||x|^{-1}|u|^5\|_{L_{\ist}^2 L_x^\frac65} + \||u|^4|\partial u|\|_{L_{\ist}^2 L_x^\frac65}\\
&\lesssim \Big(\|u\|_{L_{\ist}^{10} L_x^{10}}^4\|\frac{|u|}{|x|}\|_{L_{\ist}^{10} L_x^\frac{30}{13}} + \|u\|_{L_{\ist}^{10} L_x^{10}}^4\|\partial u\|_{L_{\ist}^{10} L_x^\frac{30}{13} }\Big)\\
&\lesssim \|u\|_{L_{\ist}^{10} H_\frac{30}{10}^{1}}^5 \lesssim \rho^5.
\end{align*}
 Hence we obtain
$$
\|\Psi(u)\|_{S_T^{1, \;0}} \le \|e^{it\Delta}\varphi\|_{S_T^{1, \;0}} + C(1+T^\frac12)(\rho^3 + \rho^5).
$$
The choice of $T = T(\varphi)$ and $\rho$ such that $\|e^{it\Delta}\varphi\|_{S_T^{1, \;0}} \le \rho/2$ and $C(1+T^\frac12)(\rho^3 + \rho^5) \le \rho/2$ shows the self-mapping of $\Psi$ from $\xloz(T, \rho)$ to $\xloz(T, \rho)$. We can also readily show that for a little smaller $T$
$$
d^{1,\; 0}(\Psi(u), \Psi(v)) \le \frac12d^{1, \;0}(u, v),
$$
because we have only to replace a $u$  with $u - v$ in the proof of self-mapping. Then the local well-posedness in $H^1$ is clear from the contraction.

\subsection{Case: $0 < b_1 < 1$}\label{case1}

Given $\rho$, from Lemmas \ref{sobo-sp}, \ref{hs-ineq}, \ref{ang-decay}, and \ref{str} we obtain that for any $u \in \xloo(T, \rho)$
\begin{align*}
&\|N_1^{1, 1}\|_{L_{\ist}^{10}L_x^\frac{30}{13} \cap L_{I_T}^\infty L_x^2}\\
 &\lesssim \||x|^{b_1-1}|u|^3\|_{L_{\ist}^1 L_x^2} + \||x|^{b_1}|u|^2|\partial u|\|_{L_{\ist}^1 L_x^2} + \||x|^{b_1-1}|u|^2|\mathbf Lu|\|_{{L_{\ist}^2 L_x^\frac65}}\\
&\qquad\qquad + \||x|^{b_1}|u||\partial u||\mathbf L u|\|_{L_{\ist}^2 L_x^\frac65} + \||x|^{b_1}|u|^2|\partial \mathbf L u|\|_{L_{\ist}^1 L_x^2}\\
&\lesssim T\Big(\||x|^\frac{b_1}2u\|_{L_{\ist}^\infty L_x^\infty}^2\|\frac{|u|}{|x|}\|_{L_{\ist}^\infty L_x^2} + \||x|^\frac{b_1}2u\|_{L_{\ist}^\infty L_x^\infty}^2\|\partial u\|_{L_{\ist}^\infty L_x^2 } + \||x|^\frac{b_1}2 u\|_{L_{\ist}^\infty L_x^\infty}^2\|\partial \mathbf L u\|_{L_{\ist}^\infty L_x^2}\Big)\\
&\quad + T^\frac12\Big(\|\frac{|u|}{|x|^{1-b_1}}\|_{L_{\ist}^\infty L_x^2}\|u\|_{L_{\ist}^\infty L_x^6}\|\mathbf L u\|_{L_{\ist}^\infty L_x^6} + \||x|^{b_1}u\|_{L_{\ist}^\infty L_x^\infty}\|\partial u\|_{L_{\ist}^\infty L_x^2}\|\mathbf L u\|_{L_{\ist}^\infty L_x^3} \Big)\\
&\lesssim (T+T^\frac12)\|u\|_{L_{\ist}^\infty \hl^{1, 1}}^3 \lesssim (T+T^\frac12)\rho^3.
\end{align*}
On the other hand, $N_2^{1, 1}$ consists of $u, \partial u$, $\mathbf L u$, $\partial \mathbf L u$, and additional $|u|^2$. For simplicity we only consider $\||x|^{b_2}|u|^3|\partial u||\mathbf L u|\|_{L_{\ist}^2 L_x^\frac65}$.
If $b_2 > 0$, then
$$
\||x|^{b_2}|u|^3|\partial u||\mathbf L u|\|_{L_{\ist}^2 L_x^\frac65} \lesssim T^\frac12\||x|^\frac{b_2}3|u|\|_{L_{\ist}^\infty L_x^\infty }^3\|\partial u\|_{L_{\ist}^\infty L_x^2}\|\mathbf L u\|_{L_{\ist}^\infty L_x^3} \lesssim T^\frac12\|u\|_{L_{\ist}^\infty \hl^{1, 1}}^5 \lesssim T^\frac12\rho^5.
$$
If $b_2 = 0$, then
$$
\||u|^3|\partial u||\mathbf L u|\|_{L_{\ist}^2 L_x^\frac65} = \|u\|_{L_{\ist}^{10} L_x^{10}}^3\|\partial u\|_{L_{\ist}^{10} L_x^{\frac{30}{13}}}\|\mathbf L u\|_{L_{\ist}^{10} L_x^{10}} \lesssim \|u\|_{L_{\ist}^\infty H_{\mathbf L,\; \frac{30}{13}}^{1,\; 1}}^5 \lesssim \rho^5.
$$
Hence we obtain that for $b_2 > 0$
$$
\|\Psi(u)\|_{S_T^{1,\; 1}} \lesssim \|\varphi\|_{\hl^{1, 1}} + (T+T^\frac12)(\rho^3 + \rho^5)
$$
and for $b_2 = 0$
$$
\|\Psi(u)\|_{S_T^{1,\; 1}} \lesssim \|e^{it\Delta}\varphi\|_{S_T^{1,\; 1}} + (1+T^\frac12)(\rho^3 + \rho^5).
$$
Now we can choose $T$ and $\rho$ so that $\Psi$ becomes self-mapping from $\xloo(T, \rho)$ to $\xloo(T, \rho)$, and also choose  a little smaller $T$ so that
$$
d^{1, \;1}(\Psi(u), \Psi(v)) \le \frac12d^{1, \;1}(u, v).
$$
This completes the proof of part $(2)$ of Theorem \ref{lwp}.

\subsection{Case: $b_1 = 1$}

In view of the proof in Section \ref{case1} we have only to estimate $N_l^{1,\; 2}$ for the contraction on $\xlot(T, \rho)$.
From \eqref{k} we get
\begin{align*}
&\|N_1^{1, \;2}\|_{L_{\ist}^{10}L_x^\frac{30}{13} \cap L_{I_T}^\infty L_x^2}\\
 &\lesssim \||u|^3\|_{L_{\ist}^1 L_x^2} + \||x||u|^2|\partial u|\|_{L_{\ist}^1 L_x^2} + \||u|^2|\mathbf L u|\|_{L_{\ist}^1 L_x^2} + \||u||\mathbf L u|^2\|_{{L_{\ist}^1 L_x^2}} + \||u|^2||\mathbf L|^2u|\|_{{L_{\ist}^1 L_x^2}}\\
 &\quad + \||x||\partial u||\mathbf L u|^2\|_{{L_{\ist}^1 L_x^2}} + \||x||u|^2|\partial \mathbf Lu|\|_{{L_{\ist}^1 L_x^2}} + \||x||u||\mathbf L u||\partial \mathbf Lu|\|_{{L_{\ist}^1 L_x^2}} + \||x||u||\partial u|||\mathbf L|^2 u|\|_{{L_{\ist}^2 L_x^\frac65}}\\
 &\quad + \||x||u|^2|\partial |\mathbf L|^2u|\|_{{L_{\ist}^1 L_x^2}}\\
&\lesssim T\Big(\|u\|_{L_{\ist}^\infty L_x^6}^3 + \||x|^\frac12u\|_{L_{\ist}^\infty L_x^\infty}^2\|\partial u\|_{L_{\ist}^\infty L_x^2 } + \|u\|_{L_{\ist}^\infty L_x^6}^2\|\mathbf L u\|_{L_{\ist}^\infty L_x^6} + \|u\|_{L_{\ist}^\infty L_x^6}\|\mathbf L u\|_{L_{\ist}^\infty L_x^6}^2\\
&\qquad\qquad + \|u\|_{L_{\ist}^\infty L_x^6}^2\||\mathbf L|^2 u\|_{L_{\ist}^\infty L_x^6} + \|\partial u\|_{L_{\ist}^\infty L_x^2}\||x|^\frac12\mathbf L u\|_{L_{\ist}^\infty L_x^\infty}^2 + \||x|^\frac12 u\|_{L_{\ist}^\infty L_x^\infty}\||x|^\frac12\mathbf L u\|_{L_{\ist}^\infty L_x^\infty}\|\partial \mathbf L u\|_{L_{\ist}^\infty L_x^2 }\Big)\\
&\qquad\qquad + T^\frac12\||x| u\|_{L_{\ist}^\infty L_x^\infty}\|\partial u\|_{L_{\ist}^\infty L_x^2}\||\mathbf L|^2 u\|_{L_{\ist}^\infty L_x^3} + T\||x|^\frac12 u\|_{L_{\ist}^\infty L_x^\infty}^2\|\partial |\mathbf L|^2u\|_{L_{\ist}^\infty L_x^2}\\
&\lesssim (T+T^\frac12)\rho^3.
\end{align*}
Similarly we obtain
$$
\|N_2^{1,\; 2}\|_{L_{\ist}^{10}L_x^\frac{30}{13} \cap L_{I_T}^\infty L_x^2} \lesssim \left\{\begin{array}{l}(T + T^\frac12)\rho^5\;\; \mbox{if}\;\; b_2 > 0,\\
\rho^5\;\; \mbox{if}\;\; b_2 = 0.\end{array}\right.
$$

\subsection{Case: $1 < b_1 < 2$}
In this case we use a modified complete metric space $\xl^{1, 2, 1}(T, \rho) = \xlot(T, \rho) \cap \xlto(T, \rho)$ with metric $d^{1, 2, 1} = d^{1, 2} + d^{2, 1}$.
To show the contraction on $\xl^{1, 2, 1}(T, \rho)$ we consider $N_l^{2, 1}$ and $N_l^{1, 2}$.
Using \eqref{k}, and Lemmas \ref{ang-sobo} and \ref{hs-ineq}, we have
\begin{align*}
&\|N_1^{1, 2}\|_{L_{\ist}^{10}L_x^\frac{30}{13} \cap L_{I_T}^\infty L_x^2}\\
 &\lesssim \||x|^{b_1-1}|u|^3\|_{L_{\ist}^1 L_x^2} + \||x|^{b_1}|u|^2|\partial u|\|_{L_{\ist}^1 L_x^2} + \||x|^{b_1-1}|u|^2|\mathbf L u|\|_{L_{\ist}^1 L_x^2} + \||x|^{b_1-1}|u||\mathbf L u|^2\|_{L_{\ist}^1 L_x^2}\\
 &\quad  + \||x|^{b_1-1}|u|^2||\mathbf L|^2u|\|_{L_{\ist}^1 L_x^2} + \||x|^{b_1}|\partial u||\mathbf L u|^2\|_{L_{\ist}^1 L_x^2} + \||x|^{b_1}|u||\mathbf L u||\partial \mathbf Lu|\|_{L_{\ist}^1 L_x^2} \\
 &\quad + \||x|^{b_1}|u||\partial u|||\mathbf L|^2 u|\|_{L_{\ist}^1 L_x^2} + \||x|^{b_1}|u|^2|\partial |\mathbf L|^2u|\|_{{L_{\ist}^1 L_x^2}}\\
&\lesssim T\Big(\||x|^{b_1-1}u\|_{L_{\ist}^\infty L_x^\infty}\|u\|_{L_{\ist}^\infty L_x^4}^2 + \||x|^\frac{b_1}2u\|_{L_{\ist}^\infty L_x^\infty}^2\|\partial u\|_{L_{\ist}^\infty L_x^2 } + \||x|^\frac{b_1-1}2u\|_{L_{\ist}^\infty L_x^\infty}^2\|\mathbf L u\|_{L_{\ist}^\infty L_x^2}\\
&\quad  + \||x|^{b_1-1}u\|_{L_{\ist}^\infty L_x^\infty}\|\mathbf L u\|_{L_{\ist}^\infty L_x^4}^2 + \||x|^\frac{b_1-1}2u\|_{L_{\ist}^\infty L_x^\infty}^2\||\mathbf L|^2 u\|_{L_{\ist}^\infty L_x^2} + \|\partial u\|_{L_{\ist}^\infty L_x^2}\||x|^\frac{b_1}2\mathbf L u\|_{L_{\ist}^\infty L_x^\infty}^2\\
&\quad +\||x|^\frac{b_1}2 u\|_{L_{\ist}^\infty L_x^\infty}\||x|^\frac{b_1}2\mathbf L u\|_{L_{\ist}^\infty L_x^\infty}\|\partial \mathbf L u\|_{L_{\ist}^\infty L_x^2 } + \||x|^\frac{b_1}2 u\|_{L_{\ist}^\infty L_x^\infty}\||x|^\frac{b_1}2\partial u\|_{L_{\ist}^\infty L_x^\infty}\||\mathbf L|^2 u\|_{L_{\ist}^\infty L_x^2}\\
&\quad + \||x|^\frac{b_1}2 u\|_{L_{\ist}^\infty L_x^\infty}^2\|\partial |\mathbf L|^2u\|_{L_{\ist}^\infty L_x^2}\Big)\\
&\lesssim T\|u\|_{L_{\ist}^\infty (\hl^{1, 2}\cap \hl^{2,1})}^3\\
& \lesssim T\rho^3
\end{align*}
and
\begin{align*}
&\|N_1^{2, 1}\|_{L_{\ist}^{10}L_x^\frac{30}{13} \cap L_{I_T}^\infty L_x^2}\\
 &\lesssim \||x|^{b_1-2}|u|^3\|_{L_{\ist}^1 L_x^2} + \||x|^{b_1-1}|u|^2|\partial u|\|_{L_{\ist}^1 L_x^2} + \||x|^{b_1}|u||\partial u|^2\|_{L_{\ist}^1 L_x^2} + \||x|^{b_1}|u|^2|\partial^2 u|\|_{L_{\ist}^1 L_x^2}\\
 &\quad + \||x|^{b_1-2}|u|^2|\mathbf Lu|\|_{{L_{\ist}^1 L_x^2}} + \||x|^{b_1-1}|u||\partial u||\mathbf L u|\|_{L_{\ist}^1 L_x^2} + \||x|^{b_1-1}|u|^2|\partial \mathbf L u|\|_{L_{\ist}^1 L_x^2}\\
 &\quad + \||x|^{b_1}|u||\partial^2u||\mathbf Lu|\|_{{L_{\ist}^1 L_x^2}} + \||x|^{b_1}|\partial u|^2|\mathbf Lu|\|_{{L_{\ist}^1 L_x^2}}  + \||x|^{b_1}|u||\partial u||\partial \mathbf Lu|\|_{{L_{\ist}^1 L_x^2}}\\
  &\quad + \||x|^{b_1}|u|^2|\partial^2 \mathbf Lu|\|_{{L_{\ist}^1 L_x^2}}\\
&\lesssim T\Big(\||x|^\frac{b_1-1}2u\|_{L_{\ist}^\infty L_x^\infty}^2\|\frac{|u|}{|x|}\|_{L_{\ist}L_x^2} + \||x|^\frac{b_1-1}2u\|_{L_{\ist}^\infty L_x^\infty}^2\|\partial u\|_{L_{\ist}^\infty L_x^2 } \\
&\quad  + \||x|^\frac{b_1}2 u\|_{L_{\ist}^\infty L_x^\infty}\||x|^\frac{b_1}2 \partial u\|_{L_{\ist}^\infty L_x^\infty }\|\partial u\|_{L_{\ist}^\infty L_x^2 } + \||x|^\frac{b_1}2 u\|_{L_{\ist}^\infty L_x^\infty}^2\|\partial^2 u\|_{L_{\ist}^\infty L_x^2 } \\
&\quad + \||x|^{b_1-2}|u|\|_{L_{\ist}^\infty L_x^2}\|u\|_{L_{\ist}^\infty L_x^\infty}\|\mathbf L u\|_{L_{\ist}^\infty L_x^\infty} + \||x|^{b_1-1}u\|_{L_{\ist}^\infty L_x^\infty}\|\partial u\|_{L_{\ist}^\infty L_x^2}\|\mathbf L u\|_{L_{\ist}^\infty L_x^\infty}\\
&\quad  + \||x|^{b_1-1}u\|_{L_{\ist}^\infty L_x^\infty}\| u\|_{L_{\ist}^\infty L_x^\infty}\|\partial \mathbf L u\|_{L_{\ist}^\infty L_x^2} + \||x|^\frac{b_1}2u\|_{L_{\ist}^\infty L_x^\infty}\|\partial^2 u\|_{L_{\ist}^\infty L_x^2}\||x|^\frac{b_1}2 \mathbf L u\|_{L_{\ist}^\infty L_x^\infty } \\
&\quad  +\||x|^\frac{b_1}2\partial u\|_{L_{\ist}^\infty L_x^\infty}\||x|^\frac{b_1}2\partial u\|_{L_{\ist}^\infty L_x^\infty}\| \mathbf L u\|_{L_{\ist}^\infty L_x^2} + \||x|^\frac{b_1}2 u\|_{L_{\ist}^\infty L_x^\infty}\||x|^\frac{b_1}2\partial u\|_{L_{\ist}^\infty L_x^\infty}\|\partial \mathbf L u\|_{L_{\ist}^\infty L_x^2}\\
&\quad  + \||x|^\frac{b_1}2 u\|_{L_{\ist}^\infty L_x^\infty}\||x|^\frac{b_1}2 u\|_{L_{\ist}^\infty L_x^\infty}\|\partial^2 \mathbf L u\|_{L_{\ist}^\infty L_x^2} \Big)\\
&\lesssim T\|u\|_{L_{\ist}^\infty (\hl^{1, 2}\cap \hl^{2,1})}^3\\
& \lesssim T\rho^3.
\end{align*}
Also we have $\|N_2^{2, 1}\|_{L_{\ist}^{10}L_x^\frac{30}{13} \cap L_{\ist}^\infty L_x^2} \lesssim T \rho^5$ for $b_2 > 0$ and $\lesssim \rho^5$ for $b_2 = 0$.

\subsection{Case: $b_1 = 2$}
As above we consider $N_l^{2, 2}$.
Together with Lemmas \ref{str},  \ref{hs-ineq}, and \ref{ang-sobo}, the bound \eqref{k} of $K_\ell$ gives us
\begin{align*}
&\|N_1^{2, 2}\|_{L_{\ist}^{10}L_x^\frac{30}{13} \cap L_{I_T}^\infty L_x^2}\\
 &\lesssim \||u|^3\|_{L_{\ist}^1 L_x^2} + \||x||u|^2|\partial u|\|_{L_{\ist}^1 L_x^2} + \||x|^2|u||\partial u|^2\|_{L_{\ist}^1 L_x^2} + \||x|^2|u|^2|\partial^2 u|\|_{L_{\ist}^1 L_x^2}\\
 &\quad + \||u|^2|\mathbf Lu|\|_{{L_{\ist}^1 L_x^2}} + \||x||u||\partial u||\mathbf L u|\|_{L_{\ist}^1 L_x^2} + \||x||u|^2|\partial \mathbf L u|\|_{L_{\ist}^1 L_x^2} + \||x|^2|u||\partial^2u||\mathbf Lu|\|_{{L_{\ist}^1 L_x^2}}\\
 &\quad  + \||x|^2|\partial u|^2|\mathbf Lu|\|_{{L_{\ist}^1 L_x^2}} + \||x|^2|u||\partial u||\partial \mathbf Lu|\|_{{L_{\ist}^1 L_x^2}} + \||x|^2|u|^2(|\partial^2 \mathbf Lu| + |\partial^2|\mathbf L|^2u|)\|_{{L_{\ist}^1 L_x^2}}\\
&\lesssim T\Big(\|u\|_{L_{\ist}^\infty L_x^6}^3 + \||x|^\frac12u\|_{L_{\ist}^\infty L_x^\infty}^2\|\partial u\|_{L_{\ist}^\infty L_x^2 } + \||x| u\|_{L_{\ist}^\infty L_x^\infty}\||x| \partial u\|_{L_{\ist}^\infty L_x^\infty }\|\partial u\|_{L_{\ist}^\infty L_x^2 } \\
&\quad + \||x| u\|_{L_{\ist}^\infty L_x^\infty}^2\|\partial^2 u\|_{L_{\ist}^\infty L_x^2 } + \|u\|_{L_{\ist}^\infty L_x^\infty}^2\|\mathbf L u\|_{L_{\ist}^\infty L_x^\infty} + \||x|u\|_{L_{\ist}^\infty L_x^\infty}\|\partial u\|_{L_{\ist}^\infty L_x^2}\|\mathbf L u\|_{L_{\ist}^\infty L_x^\infty}\\
&\quad + \||x|u\|_{L_{\ist}^\infty L_x^\infty}\| u\|_{L_{\ist}^\infty L_x^\infty}\|\partial \mathbf L u\|_{L_{\ist}^\infty L_x^2} + \||x|u\|_{L_{\ist}^\infty L_x^\infty}\|\partial^2 u\|_{L_{\ist}^\infty L_x^2}\||x| \mathbf L u\|_{L_{\ist}^\infty L_x^\infty }\\
&\quad +\||x|\partial u\|_{L_{\ist}^\infty L_x^\infty}\||x|\partial u\|_{L_{\ist}^\infty L_x^\infty}\| \mathbf L u\|_{L_{\ist}^\infty L_x^2}  + \||x| u\|_{L_{\ist}^\infty L_x^\infty}\||x|\partial u\|_{L_{\ist}^\infty L_x^\infty}\|\partial \mathbf L u\|_{L_{\ist}^\infty L_x^2} \\
&\quad + \||x| u\|_{L_{\ist}^\infty L_x^\infty}^2(\|\partial^2 \mathbf L u\|_{L_{\ist}^\infty L_x^2} +  \|\partial^2 |\mathbf L|^2 u\|_{L_{\ist}^\infty L_x^2}\Big)\\
&\lesssim T\|u\|_{L_{\ist}^\infty \hl^{1, 2}}^3\\
& \lesssim T\rho^3
\end{align*}
and $\|N_2^{2, 2}\|_{L_{\ist}^{10}L_x^\frac{30}{13} \cap L_{I_T}^\infty L_x^2} \lesssim T\rho^5$  for $b_2 > 0$ and $\lesssim \rho^5$ for $b_2 = 0$.

\subsection{Mass and energy conservation}
According to the nonlinear estimates above, one can readily show that if $\varphi \in \hl^{2, 2}$ (or $\in H^2$) then the solution $u \in C(I_*; \hl^{2, 2})$ for $b_1 > 0$ (or $\in C(I_*; H^2)$ for $b_l = 0$, respectively).
So we first assume that $\varphi \in \hl^{2, 2}$ ( or $\in H^2$). Then the map $g(u) = K_1Q_1(u) + K_2Q_2(u) \in C(\hl^{2,\; 2}, L_x^2)$. Hence for any $I_T \subset I_*$ if $u \in C(I_T; \hl^{2,\; 2})$ (or $\in C(I_T; H^2)$), then $u_t \in C(I_T; L_x^2)$. The mass or energy conservation follows from $\hl^{2, \;2}$ (or $H^2$) regularity. By continuous dependency and standard limiting argument for the sequence $\varphi_k \in \hl^{2,\; 2}$ (or $H^2$) with $\varphi_k \to \varphi$ in $\hl^{1, 1}$ or $\hl^{1,\; 2}$ (or $H^1$, respectively), we get the mass and energy conservation in the case that $b_1 > 0$ (or $b_l = 0$).

\section{Proof of Blowup}

\newcommand{\te}{\theta_{\varepsilon}}
\newcommand{\fe}{f_\varepsilon}
\newcommand{\ma}{\mathcal{A}}

% and the generator of Galilean transformation $\mathbf J = e^{-it\Delta} x e^{it\Delta} = x + 2it\nabla$
%, and that $|\mathbf J|^2:= \mathbf J\cdot \mathbf J = |x|^2 + 4tA - 4t^2\Delta$, where $A = \frac  i2(x\cdot \nabla + \nabla \cdot x)$ is the self-adjoint dilation operator
%Let $I$ be open interval containing $0$ and define
%\begin{align*}
%\hj^1(I) &= \{F \in C(\overline{I}; H^1): \mathbf J F \in C(\overline{I}; L_x^2)\}, \qquad\qquad\|F\|_{\hj^1(I)} : = \|F\|_{L_{I}^\infty H^1} + \|\mathbf Jf\|_{L_{I}^\infty L_x^2}.\\
%\hj^2(I) &= \{F \in C(\overline{I}; H^2): AF, |\mathbf J|^2 F \in C(\overline{I}; L_x^2)\}, \quad\|F\|_{\hj^2(I)} : = \|F\|_{L_{I}^\infty H^2} + \|AF\|_{L_{I}^\infty L_x^2} + \||\mathbf J^2 F\|_{L_{I}^\infty L_x^2}.
%\end{align*}
%In case that $T = \infty$, then we denote $\hj^m(\infty)$ by $\hj^m$.

We show the finite time blowup via standard virial argument. To avoid duplication of proof we consider the case $b_1 > 0$. For the case of constant $K_l$ see \cite{caz}.
\begin{lem}
Let $\varphi \in \hl^{1, 1}$ and $x\varphi \in L_x^2$, and let $u$ be the solution of \eqref{main eqn} in $C([-T, T]; \hl^{1, 1})$. Then $xu \in C([-T,T]; L_x^2)$ and it satisfies that
\begin{align}\label{virial}
\|||x|u(t)\|_{L_x^2}^2 = \||x|\varphi\|_{L_x^2}^2 + 4\int_0^t\ma(s)\,ds,
\end{align}
where $\ma(t) = {\rm Im}\int \overline u(t) x\cdot \nabla u(t)\,dx$.
\begin{align}\label{dilation}
\ma(t) = 4tE(\varphi) + \frac12\int_0^t\!\!\!\int (K_1 - x\cdot \nabla K_1)|u|^4\,dxds + \frac13\int_0^t\!\!\!\int (4K_2- x\cdot \nabla K_2) |u|^6\,dxds.
\end{align}
\end{lem}
\begin{proof}
Let $\te(x) = e^{-\varepsilon |x|^2}$ and $\fe(t) = \| \te(x)|x|u(t)\|_{L_x^2}^2\,dx$. Then since $g(u) = \sum_{l = 1,2}K_l Q_l(u) \in C([-T, T], L_x^\frac65)$, by direct differentiation one can easily obtain that
\begin{align*}
\fe(t) = \fe(0) + 4\int_0^t {\rm Im}\int \te(x)(1-2\varepsilon |x|^2)\te(x) \overline u x\cdot \nabla u,dxdt,\\
\end{align*}
and thus
$$
 \sqrt{\fe(t)} \le \||x|\varphi\|_{L_x^2} + C\int_0^t \|\nabla u(t)\|_{L_x^2}\,ds.
$$
Using Fatou's lemma, we obtain that $xu(t) \in (L^\infty \cap C)([-T,T]; L_x^2)$ and \eqref{virial}.

Here one can also show that if a sequence $\{\varphi_k\} \subset  \hl^{2, 2}$ satisfies that $\varphi_k \to \varphi$ in $\hl^{n, \ell}$ and $x\varphi_k \to x\varphi$ in $L_x^2$, then the solution sequence $\{u_k\}$ satisfies that
\begin{align}\label{weight-conv}
xu_k \to xu \quad {\rm in} \quad (L^\infty \cap C)([-T, T]; L_x^2).
\end{align}

\newcommand{\mae}{\mathcal{A}_\varepsilon}
Due to the continuous dependency on the initial data and \eqref{weight-conv} we may assume that $\varphi \in H^2$ and $u \in C([-T, T]; H^2 \cap \hl^{n, \;\ell}) \cap C^1([-T, T]; L_x^2)$ and $xu \in C([-T, T]; L_x^2)$.
Let us consider a modified quantity ${\rm Im}\int \te \overline u x\cdot \nabla u\,dx$. Then the identity \eqref{dilation} follow from direct differentiation of this quantity and standard limiting argument $\varepsilon \to 0$.

Now from \eqref{virial} and \eqref{dilation}, and from the condition of $K_l$ it follows that
\begin{align}\label{concavity}
\||x|u(t)\|_{L_x^2}^2 \le \||x|\varphi\|_{L_x^2}^2 + 4t{\rm Im}\int\overline{\varphi} x\cdot \nabla \varphi\,dx + (8+4\alpha)t^2 E(\varphi).
\end{align}
Since $E(\varphi) < 0$, \eqref{concavity} gives us the finite time blowup.
\end{proof}

\section{Scattering: Proof of Theorem \ref{scattering}}

\subsection{Nonlinear estimates}

\begin{lem}\label{cubic-est}
Let $0 \le b_1 < \frac23$. Then we have for any $f \in H_{\mathbf L}^{1, 1}$, $g \in H^1$, and $h \in L_x^6$ we have
$$
\||x|^{b_1-1}fgh\|_{L_x^\frac65} + \||x|^{b_1}fgh\|_{L_x^\frac65} \lesssim \|f\|_{H_{\mathbf L}^{1, 1}}\|g\|_{ H^1}\|h\|_{L_x^6}.
$$
\end{lem}
\begin{proof}
For the first term we have from Lemma \ref{hs-ineq} that
$$
\||x|^{b_1-1}fgh\|_{L_x^\frac65} \lesssim \||x|^{b_1}f\|_{L_x^6}\||x|^{-1}g\|_{L_x^2}\|h\|_{L_x^6} \lesssim \||x|^{b_1}f\|_{L^6}\|g\|_{H^1}\|h\|_{L_x^6}.
$$
If $b_1 = 0$, then we are done by Sobolev embedding. If $b_1> 0$, then let us choose $\varepsilon$ such that $\frac13 < \varepsilon < 1 - b_1$ and set $p = 6\varepsilon$. Then by Lemma \ref{ang-decay} with $b = b_1$ and $p = 6\varepsilon$ we have
$$
\||x|^{b_1}f\|_{L_x^6} \lesssim \|f\|_{H_{\mathbf L}^{1, 1}}^{1-\varepsilon}\|f\|_{L_x^p}^\varepsilon.
$$
Since $2 < p < 6 $, Sobolev embedding gives the desired bound.

By the same way we can treat the second term as follows. If $b_1 > 0$
$$
\||x|^{b_1}fgh\|_{L_x^\frac65} \lesssim \||x|^{b_1}f\|_{L_x^6}\|g\|_{L_x^2}\|h\|_{L_x^6} \lesssim \|f\|_{H_{\mathbf L}^{1, 1}}\|g\|_{ H^1}\|h\|_{L_x^6}.
$$
If $b_1 = 0$, then for small positive $\varepsilon$ we have
$$
\|fgh\|_{L_x^\frac65} \lesssim \||x|^{\varepsilon}f\|_{L_x^6}\||x|^{-\varepsilon}g\|_{L_x^2}\|h\|_{L_x^6} \lesssim \|f\|_{H_{\mathbf L}^{1, 1}}\|g\|_{ H^1}\|h\|_{L_x^6}.
$$

\end{proof}

\begin{lem}\label{quintic-est}
$(1)$ Let $0 \le b_2 < \frac83$. Then we have for any $f_1, f_2, f_2 \in H_{\mathbf L}^{1, 1}$, $g \in H^1$, and $h \in L_x^6$ we have
$$
\||x|^{b_2-1}f_1f_2f_3gh\|_{L_x^\frac65} + \||x|^{b_2}f_1f_2f_3gh\|_{L_x^\frac65} \lesssim \prod_{j = 1}^3\|f_j\|_{H_{\mathbf L}^{1, 1}}\|g\|_{ H^1}\|h\|_{L_x^6}.
$$
$(2)$ If $b_2 = 0$, then we have for any $f_1, f_2, f_2 \in H_{\mathbf L, \frac{30}{13}}^{1, 1}$, $g \in H_{\frac{30}{13}}^1$, and $h \in L_x^\frac{30}{13}$ we have
$$
\||x|^{-1}f_1f_2f_3gh\|_{L_x^\frac65} + \|f_1f_2f_3gh\|_{L_x^\frac65} \lesssim \prod_{j = 1}^3\|f_j\|_{H_{\mathbf L, \frac{30}{13}}^{1, 1}}\|g\|_{ H_{\frac{30}{13}}^1}\|h\|_{L_x^\frac{30}{13}}.
$$

\end{lem}

\begin{proof} We first consider the case $b_2 > 0$.
Choose $\frac19 \le \varepsilon < \min(\frac13, 1- \frac{ b_2}3)$. Then from Lemma \ref{ang-decay} with $b = \frac{b_2}3, p = 18\varepsilon$ and Lemma \ref{hs-ineq} it follows that
\begin{align*}
\||x|^{b_2-1}f_1f_2f_3gh\|_{L_x^\frac65} &\lesssim \prod_{j = 1}^3\||x|^{\frac{b_2}3}f_j\|_{L_x^{18}}\||x|^{-1}g\|_{L_x^2}\|h\|_{L_x^6}\\
&\lesssim \prod_{j = 1}^3\|f_j\|_{H_{\mathbf L}^{1, 1}}^{3(1-\varepsilon)}\|f_j\|_{L_x^p}^{3\varepsilon}\|g\|_{ H^1}\|h\|_{L_x^6}.
\end{align*}
Since $2 \le p < 6$, Sobolev embedding ($H^1 \hookrightarrow L^p$) leads us to the desired estimate.

On the other hand, one can easily see that
\begin{align*}
\||x|^{b_2}f_1f_2f_3gh\|_{L_x^\frac65}  &\lesssim \prod_{j = 1}^3\||x|^{\frac{b_2}3}f_j\|_{L_x^{18}}\|g\|_{L_x^2}\|h\|_{L_x^6}\\
& \lesssim \prod_{j = 1}^3\|f_j\|_{H_{\mathbf L}^{1, 1}}\|g\|_{ H^1}\|h\|_{L_x^6}.
\end{align*}

If $b_2 = 0$, then we have
\begin{align*}
\||x|^{-1}f_1f_2f_3gh\|_{L_x^\frac65} + \|f_1f_2f_3gh\|_{L_x^\frac65} &\lesssim \prod_{j = 1}^3\|f_j\|_{L_x^{10}}(\||x|^{-1}g\|_{L_x^\frac{30}{13}} + \|g\|_{L_x^\frac{30}{13}})\|h\|_{L_x^\frac{30}{13}}\\
&\lesssim \prod_{j = 1}^3\|f_j\|_{H_{\mathbf L,\;\frac{30}{13}}^{1, 1}}^3\|g\|_{ H_{\frac{30}{13}}^1}\|h\|_{L_x^\frac{30}{13}}.
\end{align*}

\end{proof}

\begin{rem}\label{gen-case}
We can apply the above estimate to non-algebraic cases $|x|^b|f|^\alpha gh$ with $\frac13 < \alpha < 3, 0 < b < \alpha - \frac13$). By taking $\frac1{3\alpha} \le \varepsilon < 1- \frac b\alpha$, one can get
\begin{align*}
\||x|^{b-1}|f|^\alpha gh\|_{L_x^\frac65} &\le \||x|^b|f|^\alpha\|_{L_x^6}\||x|^{-1}g\|_{L_x^2}\|h\|_{L_x^6}\\
&\lesssim \||x|^{\frac b\alpha}f\|_{L_x^{6\alpha}}^\alpha\|g\|_{H^1}\|h\|_{L_x^6}\\
&\lesssim \|f\|_{H_{\mathbf L}^{1, 1}}^{\alpha(1-\varepsilon)}\|f\|_{L_x^{6\alpha\varepsilon}}^{\alpha\varepsilon}\|g\|_{H^1}\|h\|_{L_x^6}\\
&\lesssim \|f\|_{H_{\mathbf L}^{1, 1}}^\alpha\|g\|_{H^1}\|h\|_{L_x^6}.
\end{align*}
\end{rem}

\subsection{Proof of scattering}

Let us define a complete metric space $\xl(\rho)$ by
$$
\xl(\rho) := \Big\{u \in L_t^{10}H_{\mathbf L, \;\frac{30}{13}}^{1,\;1} \cap (C \cap L_t^\infty)(\mathbb R;  H_{\mathbf L}^{1, 1}) : \|u\|_{L_t^\infty H_{\mathbf L}^{1, 1}} + \|u\|_{L_t^2H_{\mathbf L, 6}^{1, 1}} + \| u\|_{L_t^{10}H_\frac{30}{13}^{1, 1}} \le \rho \Big\}
$$
equipped with the metric $d$ such that
$$
d(u ,v) =  \|u - v\|_{\xl} := \|u - v\|_{L_t^\infty H_{\mathbf L}} + \|u - v\|_{L_t^2H_{\mathbf L, 6}^1} + \|u-v\|_{L_t^{10}H_\frac{30}{13}^{1, 1}}.$$
Let us show that the nonlinear functional $\Psi(u) = e^{it\Delta}\varphi + \mathcal N(u)$ is a contraction on $\xl(\rho)$.
For this we have only to show
\begin{align}
&\|\mathcal N(u)\|_{\xl} \lesssim \|u\|_{\xl}^3 + \|u\|_{\xl}^5,\label{contract1}\\
&\|\mathcal N(u) - \mathcal N(u)\|_{\xl} \lesssim [(\|u\|_{\xl} + \|v\|_{\xl})^2 + (\|u\|_{\xl} + \|v\|_{\xl})^4]\|u-v\|_{\xl}.\label{contract2}
\end{align}
Clearly, $\|e^{it\Delta} \varphi\|_{\xl} \lesssim \|\varphi\|_{H_{\mathbf L}^{1,\; 1}}$ by Strichartz estimates, and thus we can find $\rho$ small enough for $\Psi$ to be a contraction mapping on $\xl(\rho)$, and for the equation \eqref{main eqn} to be globally well-posed in $H_{\mathbf L}^{1,\; 1}$.

Once \eqref{main eqn} is globally well-posed, the scattering is straightforward. In fact, let us define a scattering state $u_\pm$ with $$\varphi_\pm := \varphi + \lim_{t\to \pm\infty} e^{-it\Delta}\mathcal N(u).$$ Then we get the desired result by the duality argument:
\begin{align*}
\|u(t) - u_\pm(t)\|_{H_{\mathbf L}^{1,\;1}} &= \|(1-\Delta)^\frac12\mathbf L(u(t) - u_\pm(t))\|_{L^2} \\
&= \sup_{\|\psi\|_{L^2} \le 1}\left|\int_t^{\pm \infty}\bl  (1-\Delta)^\frac12\mathbf L[K_1Q_1(u) + K_2Q_2(u)], e^{-it'\Delta}\psi \br\,dt'\right|\\
&\le \|(1-\Delta)^\frac12 \mathbf L [K_1Q_1(u) + K_2Q_2(u)]\|_{L^2(t, \pm\infty; L_x^\frac65)}\|e^{-it'\Delta}\psi\|_{L_t^2L_x^6}\\
& \lesssim \rho^2(\|u\|_{L^2(t, \pm\infty; H_{\mathbf L,\; 6}^{1,\; 1})}+ \|u\|_{L^{10}(t, \pm\infty; H_{\frac{30}{13}^{1, \;1}})}) \to 0\quad \mbox{as}\quad t \to \pm \infty.
\end{align*}
Here $(t, \pm\infty)$ means that $(t, +\infty)$ if $t > 0$ and $(-\infty, t)$ if $t < 0$.

Now it remains to show \eqref{contract1} and \eqref{contract2}. Given $\rho$, for $u \in \xl(\rho)$ we consider $N_l^{1, \;1}$ as in Section \ref{sec4}.
From the bound \eqref{k} and the endpoint Strichartz estimate it follows that
\begin{align*}
\left\|N_l^{1, \;1}\right\|_{L_t^2L_x^6 \cap L_t^{10}L_x^\frac{30}{13}} &\lesssim \|(\partial \mathbf L K_l)Q_l(u)\|_{L_t^2L_x^\frac65} + \|\partial K_l \mathbf L Q_l(u)\|_{L_t^2L_x^\frac65} + \|\mathbf L K_l \partial Q_l(u)\|_{L_t^2L_x^\frac65} + \|K_l \partial \mathbf L Q_l(u)\|_{L_t^2L_x^\frac65}\\
&\lesssim \||x|^{b_l -1}|Q_l(u)|\|_{L_t^2L_x^\frac65} + \||x|^{b_l -1}|\mathbf L Q_l(u)|\|_{L_t^2L_x^\frac65} +  \||x|^{b_l}(|\partial Q_l(u)| + |\partial \mathbf L Q_l(u)|)\|_{L_t^2L_x^\frac65}
\end{align*}
Since $|\mathbf L Q_1(u)| \lesssim |u|^2|\mathbf L u|$, $|\mathbf L Q_1(u) \lesssim |u|^2|\partial u|$, and $|\partial \mathbf L Q_l(u)| \lesssim |u||\partial u||\mathbf L u| + |u|^2|\partial \mathbf L u|$,
applying Lemma \ref{cubic-est} with $f = u$, $g = u$ or $|\mathbf L u|$, and $h = u, \partial u,$ or $|\partial \mathbf L u|$, we get
\begin{align*}
\|N_1^{1,\; 1}\|_{L_t^2L_x^6  \cap L_t^{10}L_x^\frac{30}{13}} \lesssim \|u\|_{\xl}^3.
\end{align*}
As for $Q_2(u)$, there hold
$|\mathbf L Q_2(u)| \lesssim |u|^4|\mathbf L u|$, $|\mathbf L Q_2(u) \lesssim |u|^4|\partial u|$, and $|\partial \mathbf L Q_2(u)| \lesssim |u|^3|\partial u||\mathbf L u| + |u|^4|\partial \mathbf L u|$.
Thus by taking $f_j = u$, $g = u$ or $|\mathbf L u|$, and $h = u, \partial u,$ or $|\partial \mathbf L u|$ we obtain from Lemma \ref{quintic-est}
\begin{align*}
\left\|N_2^{1,\; 1}\right\|_{L_t^2L_x^6  \cap L_t^{10}L_x^\frac{30}{13}}  \lesssim \|u\|_{\xl}^5.
\end{align*}
These show the estimate \eqref{contract1}.

To treat \eqref{contract2} let us set $w = u - v$. Then $Q_l(u) - Q_l(v)$ can be decomposed by new cubic and quintic terms of $u, v$ and only one $w$. Applying the same argument as above to these terms, one readily get the second part \eqref{contract2}. This completes the proof of Theorem \ref{scattering}.

\section{Proof of non-scattering}
We follow the argument as in \cite{ba, choz0}. By contradiction we assume that $\|\varphi_+\|_{L_x^2} \neq 0$. Since $K_l$ are real-valued, $m(u(t)) = m(\varphi)$.
We consider $H(t) = -{\rm Im}\int u(t)\overline{u_+(t)}\,dx$ for $t \gg 1$. Differentiating $H$, we get
$$
\frac{d}{dt}H(t) = {\rm Re}\int (K_1Q_1 + K_2Q_2) \overline{u_+}\,dx,
$$
where $\lambda = 0, 1$. We decompose this as follows:
$$
\frac{d}{dt}H(t) = \sum_{j = 1}^3J_1^j +  J_2,
$$
where
\begin{align*}
J_1^1 &= \int |x|^{b_1}|u_+|^4\,dx,\\
J_1^2 &= {\rm Re}\int |x|^{b_1}(|u|^2-|u_+|^2)|u_+|^2\,dx,\\
J_1^3 &= {\rm Re}\int |x|^{b_1}|u|^2(u-u_+) \overline{u_+}\,dx,
\end{align*}
and
\begin{align*}
J_2 = {\rm Re}\int |x|^{b_2}|u|^4uu_+\,dx.
\end{align*}

We estimate $J_1^1$ as follows: for $0 < \delta \ll 1 \ll k$
$$
\int_{\delta t \le |x| \le kt}|u_+|^2\,dx \le \||x|^{-\frac{b_1}2}\|_{L_x^2(\delta t \le |x| \le kt)}(J_1^1)^\frac12
\lesssim t^{\frac32-\frac{b_1}2}(J_1^1)^\frac12.
$$

It was show in \cite{ba, choz0} that $\int_{\delta t \le |x| \le kt}|u_+|^2\,dx \sim \|\varphi_+\|_{L^2}^2$ for some fixed large $k$ and small $\delta$, and for any large $t$.
From this we deduce that $$J_1^1 \gtrsim_{m(\varphi_+)} t^{-(3-b_1)}.$$

Let us denote the generator of Galilean transformation by $\mathbf J$, that is $\mathbf J = e^{-it\Delta}x e^{it\Delta}$. On the sufficiently regular function space
\begin{align}\label{j}
\mathbf J = x + 2it\nabla, \quad (\mathbf J \cdot \mathbf J)^m = (|x|^2 - 4tA - 4t^2\Delta)^m,
\end{align}
where $A$ is the self-adjoint dilation operator defined by $\frac1{2i}(x\cdot \nabla + \nabla \cdot x)$, which  yields $\mathcal A = \int \overline{u}Au\,dx$.
Since $\|u_+(t)\|_{L_x^\infty} \lesssim t^{-\frac32}\|\varphi_+\|_{L_x^1}$,
and
$$
\||x|^{2m} u_+(t)\|_{L_x^\infty} = \|e^{it\Delta}(\mathbf J \cdot \mathbf J)^{m}\varphi_+\|_{L_x^\infty} \lesssim_{\varphi_+}  t^{-(\frac32-2m)},
$$
by interpolation we see that
\begin{align}\label{t-decay}
\||x|^\theta u_+(t)\|_{L_x^\infty} \lesssim_{\varphi_+} t^{-(\frac32-\theta)}
\end{align}
for any $\theta > 0$.
By this we get $\||x|^{-\frac{b_1}3}u_+(t)\|_{L_x^\frac32(\delta t \le |x| \le kt)}(J_1^2)^\frac13  \lesssim t^{2-\frac{b_1}3}(J_1^2)^\frac13$

For $J_1^2$ we have
\begin{align*}
J_1^2 &\lesssim \||x|^{b_1}|u_+(t)|^2\|_{L_x^\infty}(\|u\|_{L_x^2} + \|u_+\|_{L_x^2})\|u - u_+\|_{L_x^2}.
\end{align*}
 Using \eqref{t-decay} we get $$|J_1^2| = o_{m(\varphi),\; \varphi_+}(t^{-(3-b_1)}).$$

To estimate $J_1^3$ and $J_2$ we need the following lemma.
\begin{lem}\label{pot-decay}
Let $u$ be a global smooth solution of \eqref{main eqn} with $K_1 = |x|^{b_1}, K_2 = |x|^{b_2}$ such that $0 \le b_2 < 3 +b_1$, and  $xu \in C(\mathbb R; L_x^2)$. Then for any large $t$ there holds
$$
V(u) := \frac14\int |x|^{b_1}|u|^4\,dx + \frac16\int |x|^{b_2}|u|^6\,dx \le C(m(\varphi), E(\varphi), \||x|\varphi\|_{L_x^2})\; t^{-(3-b_1)}.
$$
\end{lem}
From Lemma \ref{pot-decay} and inequality \eqref{t-decay} it follows that
$$
|J_1^3| \le (\int|x|^{b_1}|u|^4\,dx)^\frac12\|u - u_+\|_{L_x^2}\||x|^\frac{b_1}2u_+\|_{L_x^\infty} = o_{\varphi, \; \varphi_+}(t^{-(3-b_1)}),
$$
\begin{align*}
|J_2| &\le  \int |x|^{\frac{5b_2}6}|u|^5|x|^\frac{b_2}6|u_+|\,dx = (\int |x|^{b_2}|u|^6\,dx)^\frac56\||x|^\frac{b_2}4|u_+|\|_{L_x^\infty}^\frac23\|u_+\|_{L_x^2}^\frac13 \lesssim_{\varphi,\; \varphi_+}t^{-(\frac72 - \frac{5b_1}6-\frac{b_2}6)}\\
& =  o_{\varphi, \;\varphi_+}(t^{-(3-b_1)}) \quad(\because b_2 < 3+b_1).
\end{align*}
Therefore we conclude that for $t \gg 1$
$$
\frac{d}{dt}H(t) \gtrsim_{\varphi, \varphi_+} t^{-(3-b_1)}.
$$
Since $H(t)$ is uniformly bounded for any $t \ge 0$, the range $b_1 \ge 2$ leads us to the contradiction to the assumption $\|\varphi_+\|_{L_x^2} \neq 0$. By time symmetry a similar argument holds for negative time. We omit that part.

\begin{proof}[Proof of Lemma \ref{pot-decay}]
From \eqref{j} and \eqref{dilation} we deduce the pseudo-conformal identity:
\begin{align*}
\frac{d}{dt}\int [|\mathbf J u|^2 + 8t^2V(u)]\,dx &= -4t \left[\frac12 \int (K_1- x\cdot \nabla K_1)|u|^4\,dx + \frac13\int(4K_2-x\cdot \nabla K_2)|u|^6 \,dx\right]\\
&= -4t \left[\frac{1-b_1}2 \int |x|^{b_1}|u|^4\,dx + \frac{4-b_2}3\int|x|^{b_2}|u|^6 \,dx\right].
\end{align*}
Since $b_1 \ge 2$ and $0 \le b_2 < 3+ b_1$, by integrating over $[0, t]$ we obtain
\begin{align*}
t^2V(u(t)) &\le \frac18\||x|\varphi\|_{L_x^2}^2 + (b_1-1)\int_0^t \tau V(u(\tau))\,d\tau\\
&\le \frac18\||x|\varphi\|_{L_x^2}^2 + (b_1-1)\int_0^1 \tau V(u(\tau))\,\tau + (b_1-1)\int_1^t\tau V(u(\tau))\,d\tau\\
&\le C(m(\varphi), E(\varphi), \||x|\varphi\|_{L_x^2})) + (b_1-1)\int_1^t\tau V(u(\tau))\,d\tau.
\end{align*}
Gronwall's inequality gives us
$$
t^2V(u(t)) \le C(m(\varphi), E(\varphi), \||x|\varphi\|_{L_x^2}))\exp\left[\int_1^t \frac{b_1-1}\tau\,d\tau\right] = C(m(\varphi), E(\varphi), \||x|\varphi\|_{L_x^2}))\;t^{b_1-1}.
$$
This completes the proof of Lemma \ref{pot-decay}.
\end{proof}

%%%%%%%%%%%%%%%%%%%%%%%%%%%%%%%%%%%%%%%%%%%%%%%%%%%%%%%%%%%%%%%%%%%%%%%%%%%%%%%%%%%%%%%%%%%%%%%%%%%%%%%%%%%%%%%%%%%%%%%%%%%%%%%%%%%%%%%%%%%%%%5
%%%%%%%%%%%%%%%%%%%%%%%%%%%%%%%%%%%%%%%%%%%%%%%%%%%%%%%%%%%%%%%%%%%%%%%%%%%%%%%%%%%%%%%%%%%%%%%%%%%%%%%%%%%%%%%%%%%%%%%%%%%%%%%%%%%%%%%%%%%%%%5
\section*{Acknowledgements}
This work was supported by NRF-2018R1D1A3B07047782(Republic of Korea).
%%%%%%%%%%%%%%%%%%%%%%%%%%%%%%%%%%%%%%%%%%%%%%%%%%%%%%%%%%%%%%%%%%%%%%%%%%%%%%%%%%%%%%%%%%%%%%%%%%%%%%%%%%%%%%%%%%%%%%%%%%%%%%%%%%%%%%%%%%%%%%5
%%%%%%%%%%%%%%%%%%%%%%%%%%%%%%%%%%%%%%%%%%%%%%%%%%%%%%%%%%%%%%%%%%%%%%%%%%%%%%%%%%%%%%%%%%%%%%%%%%%%%%%%%%%%%%%%%%%%%%%%%%%%%%%%%%%%%%%%%%%%%%5

% ----------------------------------------------------------------

\end{document}